%% file: main.tex
\documentclass{amsart}
\usepackage{graphicx}
\usepackage{amsmath}
\usepackage{epstopdf}
\usepackage{amsthm}
\usepackage{color}
\usepackage{booktabs}
\usepackage{csvsimple}
\usepackage{amsfonts}
\usepackage[percent]{overpic}
\usepackage[color=green]{todonotes}
\usepackage[capitalize,nameinlink]{cleveref}
\usepackage{pc_math}
\usepackage{xcolor}
\usepackage{bbm}
\usepackage{MnSymbol}
\usepackage{algorithm}
\usepackage[algo2e,ruled,norelsize]{algorithm2e}
\usepackage[margin=1.3in]{geometry}

\newcommand{\ra}[1]{\renewcommand{\arraystretch}{#1}}

\newtheorem{thm}{Theorem}

\newtheorem{lm}{Lemma}
\newtheorem{cor}{Corrolary}

\newtheorem{dfn}{Definition}
\newtheorem{prop}{Proposition}

\crefname{algocf}{Algorithm}{Algorithms}
\crefname{thm}{Theorem}{Theorems}

\definecolor{metro_teal}{HTML}{23373b}
\definecolor{light_teal}{HTML}{7E9AA1}
\definecolor{edge_gray}{gray}{.40}

\begin{document}

\title{Configuration Models of Random Hypergraphs}

\author{Philip S. Chodrow}
\thanks{Operations Research Center, Massachusetts Institute of Technology, Cambridge, MA 02139, USA, \texttt{pchodrow@mit.edu}}

\date{\today}

\begin{abstract}
	Many empirical networks are intrinsically polyadic, with interactions occurring within groups of agents of arbitrary size. 
	There are, however, few flexible null models that can support statistical inference for such polyadic networks. 
	We define a class of null random hypergraphs that hold constant both the node degree and edge dimension sequences, generalizing the classical dyadic configuration model.
  We provide a Markov Chain Monte Carlo scheme for sampling from these models, and discuss connections and distinctions between our proposed models and previous approaches. 
	We then illustrate these models through a triplet of applications. 
	We start with two classical network topics -- triadic clustering and degree-assortativity.
	In each, we emphasize the importance of randomizing over hypergraph space rather than projected graph space, showing that this choice can dramatically alter statistical inference and study findings. 
	We then define and study the edge intersection profile of a hypergraph as a measure of higher-order correlation between edges, and derive asymptotic approximations under the stub-labeled null.
  Our experiments emphasize the ability of explicit, statistically-grounded polyadic modeling to significantly enhance the toolbox of network data science. 
  We close with suggestions for multiple avenues of future work. 
\end{abstract}

\maketitle

Graphs provide parsimonious mathematical descriptions of systems comprised of objects (nodes) and dyadic relationships (edges). 
When analyzing a given graph, a common task is to compare an observable of interest to its distribution under a suitably specified null model. 
A standard choice of null for dyadic networks is the class of \emph{configuration models} \cite{bender1978asymptotic,Bollobas1980,Molloy1998,Fosdick2018}.
Configuration models preserve the degree sequence of the graph, which counts the number of edges incident to each node.
These counts are natural first-order statistics of the graph, which are known to constrain many macroscopic graph properties \cite{Newman2001}.
Preserving these counts gives a natural null model constraint: properties observed in data that are not present in a configuration model require explanation in terms of higher-order graph structures. 

In many systems of contemporary interest, groups of arbitrary size may interact simultaneously.
Examples include social contact networks \cite{Stehle2011,Mastrandrea2015}; scholarly and professional collaboration networks; \cite{Newman2001b,Barabasi2008,Fowler2006,Porter2005}; digital communications \cite{Klimt2004}; classifications on patents \cite{Youn2015}; and many more \cite{Benson2018}.
In the past decades, the dominant approach to these systems has been to represent these networks dyadically, allowing the analyst to apply standard techniques of dyadic network science, including the configuration model. 
Recent work, however, has highlighted limitations of the dyadic paradigm in modeling of polyadic systems, both in theory \cite{Schaub2018a} and in application domains including neuroscience \cite{Giusti2016}, ecology \cite{Grilli2017}, computational social science \cite{Ugander2012a,Benson2016} among others \cite{Benson2018}.
The importance of polyadic interactions calls into question the use of the dyadic configuration model in such systems.
It is therefore desirable to construct random models for polyadic data that inherit the useful properties of the dyadic configuration model. 
In this article, we construct two such models on suitably chosen spaces of hypergraphs, and demonstrate their utility for polyadic network data science. 
Along the way, we argue for two principle theses. 
First, the choice between dyadic and polyadic null models can determine the directional findings of standard network analyses. 
Second, the use of polyadic nulls allows the analyst to measure and test rich measures of polyadic structure, thereby expanding the network-scientific toolbox. 

\subsection*{Outline}

  We begin in \Cref{sec:lit_review} with a survey of the landscape of null models for relational data, including the dyadic configuration model, random hypergraphs, and random simplicial complexes.
  In \Cref{sec:model_def}, we define stub- and vertex-labeled configuration models of random hypergraphs.
  Practical application of these models requires a sampling scheme, which we provide in \Cref{sec:sampling}.
  We turn to a triplet of illustrative applications in \Cref{sec:applications}.
  We first consider triadic closure, showing that some networks that would be considered clustered in comparison to dyadic nulls are significantly \emph{less} clustered than the corresponding hypergraph nulls.
  We then turn to degree-assortativity, where hypergraph data representations allow us to define novel measures and conduct null hypothesis tests.
  Finally, we introduce a novel measure of correlation between polyadic edges, which can be tested against either the full configuration model or analytic approximations.
  We close in \Cref{sec:conclusion} with a summary of our findings and suggestions for future development.

\section{Graphs, Hypergraphs, and Simplicial Complexes} \label{sec:lit_review}
  
  Random graph null modeling has a rich history; see \cite{Fosdick2018} for a review. 
  In this section, we take a rapid tour through some of the most important results in configuration-type models for graphs and their generalizations. 
  We begin with a brief review of the configuration model for dyadic graphs.
  \begin{dfn}[Graph]
    A \emph{graph} $G = (V,E)$ consists of a finite set $V$ of nodes or vertices and a multiset $E$ of pairs of nodes, also called \emph{edges}.
    We assume that both sets are endowed with an (arbitrary) order. 
    An edge of the form $(u,u)$ is called a \emph{self-loop}. 
    Two distinct edges $e_1$ and $e_2$ are \emph{parallel} if they are equal as sets. 
  \end{dfn}
  Let $n = \abs{V}$ and $m = \abs{E}$ be fixed.
  We denote by $\mathcal{G}^{\lcirclearrowright}$ the set of all graphs on $n$ nodes, and by $\mathcal{G} \subset \mathcal{G}^{\lcirclearrowright}$ the set of graphs on $n$ nodes without self-loops. 
  Parallel edges are permitted in $\mathcal{G}$.  
  While it is indeed possible to define configuration models on $\mathcal{G}^{\lcirclearrowright}$, we do not do so here \cite{Fosdick2018,nishimura2018connectivity}. 
  Rather we will formulate most of our results for elements of $\mathcal{G}$, only discussing $\mathcal{G}^{\lcirclearrowright}$ below in the context of certain technical issues. 
  
  The degree sequence of a graph $G = (V, E)$ is the vector $\mathbf{d}\in \mathbb{Z}^n$ defined componentwise as 
  \begin{align}
      d_v = \sum_{e \in E}\mathbbm{I}(v \in e)\;.
  \end{align}
  A configuration model is a probability distribution on the set $\mathcal{G}_\D = \{G \in \mathcal{G}: \mathrm{deg}(G) = \D\}$ of graphs with degree sequence $\D$. 
  There are two closely-related model variants which should be distinguished \cite{Fosdick2018}. 
  On its first introduction \cite{Bollobas1980}, the configuration model was defined mechanistically through a ``stub-matching'' algorithm. 
  To perform stub-matching, we place $d_v$ labeled half-edges (or ``stubs'') into an urn for each node $v$. 
  We draw half-edges two at a time, with each draw producing an edge. 
  A stub-labeled graph ``remembers'' which labeled stubs were drawn to form each edge. 
  \begin{dfn}[Stub-Labeled Graphs] \label{def:stub}
    For a fixed node set $V$ and degree sequence $\D$, define the multiset
    \begin{align*}
      \Sigma_\D = \biguplus_{v \in V}\left\{ v_1,\ldots,v_{d_v} \right\}\;,
    \end{align*}
    where $\uplus$ denotes multiset union.
    The copies $v_1,\ldots,v_{d_v}$ are called \emph{stubs} of node $v$.
    A \emph{stub-labeled graph} $S = (V,E)$ consists of the node set $V$  and an edge set $E = \{\{u_i,v_i\}\}_{i=1}^m$ which partitions $\Sigma_\D$ into unordered pairs.
    An edge of the form $\{v_i,v_j\}$ is called a \emph{self-loop}. 
  \end{dfn}
    Let $\mathcal{S}^{\lcirclearrowright}$ be the set of stub-labeled graphs, and $\mathcal{S} \subset \mathcal{S}^{\lcirclearrowright}$ the set without self-loops. 
    Technically speaking, one should remember that the set $\mathcal{S}^{\lcirclearrowright}$ of stub-labeled graphs is not a subset of the set $\mathcal{G}^{\lcirclearrowright}$ of graphs, since the objects in the edge-set are of different logical types. 
    The same is true of the sets $\mathcal{S}$ and $\mathcal{G}$. 
    These technical considerations will also apply when we generalize to hypergraphs below, but will not present any major practical issues.   
    
    There is a natural surjection $g:\mathcal{S}^{\lcirclearrowright} \rightarrow \mathcal{G}^{\lcirclearrowright}$. 
    If $S \in \mathcal{S}^{\lcirclearrowright}$, $g(S) \in \mathcal{G}^{\lcirclearrowright}$ is the graph obtained by replacing each stub $v_i$ in $S$ with $v$ and then consolidating the result as a multiset.
    We use the notation $A = g^{-1}(G)$ to refer to the preimage $A \subseteq \mathcal{S}$ of $G\subseteq \mathcal{G}$ by $g$. 
    We emphasize that $g$ is not a bijection, and the symbol $g^{-1}$ should not be interpreted as an inverse of $g$. 
    We define $\mathcal{S}_\D^{\lcirclearrowright}$ to be  $g^{-1}\left(\mathcal{G}_{\D}^{\lcirclearrowright}\right)$.  
    Note that an edge $\tilde{e} \in S$ is a self-loop if and only if $e \in g(S)$ is. 
    Because of this, $\mathcal{S} = g^{-1}\left(\mathcal{G}\right)$. 
    It is therefore natural to define $\mathcal{S}_{\D} = g^{-1}(\mathcal{G}_\D)$. 
    
  \begin{dfn}[Dyadic Configuration Models \cite{Fosdick2018}]
    Fix $\D \in \mathbb{Z}_+^n$. 
    The \emph{vertex-labeled configuration model} on $\mathcal{G}_\D$ is the uniform distribution $\eta_{\D}$. 
    Let $\lambda_{\D}$ be the uniform distribution on $\mathcal{S}_{\D}$.
    The \emph{stub-labeled configuration model} on $\mathcal{G}_\D$ is the distribution $\mu_{\D} = \lambda_{\D}\circ g^{-1}$. 
  \end{dfn}
  In our formalism, the stub-labeled configuration model is not a distribution over the space of stub-labeled graphs $\mathcal{S}_{\D}$. 
  Rather, it is the pushforward of such a distribution to the space $\mathcal{G}_\D$ of graphs. 
  Intuitively, the vertex-labeled configuration model assigns the same probability to each graph with degree sequence $\D$, while the stub-labeled model weights these graphs according to their likelihood of being realized via stub-matching. 
  One of the key insights of \cite{Bollobas1980}, since generalized by works such as \cite{Molloy1998,Angel2016}, is that these two models are related. 
  Let $\mathcal{G}_{\mathrm{simple}}$ be the set of \emph{simple graphs}, which contain neither self-loops nor parallel edges. 
  Then, $\mu_{\D}(G|G \in \mathcal{G}_{\mathrm{simple}}) = \eta_{\D}(G|G \in \mathcal{G}_{\mathrm{simple}})$. 
  Furthermore, when the degree sequence is sampled from a fixed distribution with finite second moment, $\mu_{\D}(G \in \mathcal{G}_{\mathrm{simple}})$ is bounded away from zero as $n$ grows large (see, e.g. \cite{Angel2016}), implying that repeated sampling from $\mu_{\D}$ will produce a simple graph in a number of repetitions that is asymptotically constant with respect to $n$. 
  As a result, in the ``large, sparse regime,'' it is possible to sample from the stub-labeled configuration model until a simple graph is obtained, which will then be distributed according to the vertex-labeled model. 
  This relationship is extremely convenient, enabling asymptotic analytic expressions for many quantities of theoretical and practical interest \cite{Newman2001}.
  
  This close relationship between models is likely the reason why the distinction between them has often been elided in applied network science. 
  Recently, however, the authors of \cite{Fosdick2018} pointed out that, in many data sets, these two models are not interchangeable.
  It is important to distinguish them when the data may possess multi-edges or self-loops and the edge density is relatively high. 
  The first condition is important because stub- and vertex-labeled models agree only on the subspace of simple graphs, not the full space of multigraphs. 
  The second condition locates us away from the large, sparse regime and implies that parallel edges will occur under stub-matching with non-negligible probability. 
  
  From a modeling perspective, the choice of vertex- or stub-labeling must depend on domain-specific reasoning about counterfactual comparisons. 
  Roughly, stub-labeling should be used when, for a fixed graph $G \in \mathcal{G}$, the elements of the set $g^{-1}(G) \subset \mathcal{S}$ have distinct identities in the context of the application domain. 
  This corresponds to asking whether permutations of stubs lead to meaningfully different counterfactual data sets. 
  In contrast, when stub-permutations are either nonsensical or are considered to leave the observed data unchanged, vertex-labeling is to be preferred. 
  For example, in \cite{Fosdick2018}, the authors argue that vertex-labeled nulls are most appropriate for studying a collaboration network of computational geometers. 
  Their reason is that stubs in this case correspond to an author's participation in a paper.
  It is nonsensical to say that  $A$'s first collaboration with $B$ is $B$'s second collaboration with $A$, and therefore stub-labeling is inappropriate. 

  Configuration models and their variants have played a fundamental role in the development of modern network science. 
  The seminal paper by Molloy and Reed \cite{molloy1995critical} has, according to Google Scholar, been cited at least 2,000 times since its publication, and over 800 times since 2015. 
  How can we extend these models for application to polyadic data sets? 
  A direct approach, taken in early studies such as \cite{Newman2001b}, is to compute the \emph{projected (dyadic) graph}.
  The projected graph represents each $k$-adic interaction as a $k$-clique, which contains an edge between each of the possible $\binom{k}{2}$ pairs of nodes (\Cref{fig:toy}). 
  The resulting dyadic graph may then be randomized according to vertex- or stub-labeled dyadic configuration models. 
  Projecting, however, can have unintended and occasionally counterintuitive consequences.
  First and most clearly, all properties which depend explicitly on the presence of higher-dimensional interactions are lost. 
  Second, other observables such as node degrees and edge multiplicities may be transformed in undesirable ways; for example, a single interaction between six agents becomes 15 pairwise interactions after projection. 
  As consequence, each of the six agents involved in a single $6$-adic interaction are dyadically represented as nodes of degree $5$. 
  Third, and most subtly, projecting transforms the null space for downstream hypothesis-testing in ways that may not be intended.
  For example, projecting the network in \Cref{fig:toy} prior to randomization implicitly chooses a null space of counterfactuals consisting of 17 two-author papers.
  This may be undesirable, especially when the null is viewed as a candidate data generating process. 
  Given that the data possesses higher-order interactions, a null model that is by construction unable to produce such interactions may not be physically relevant to the problem at hand. 
  
	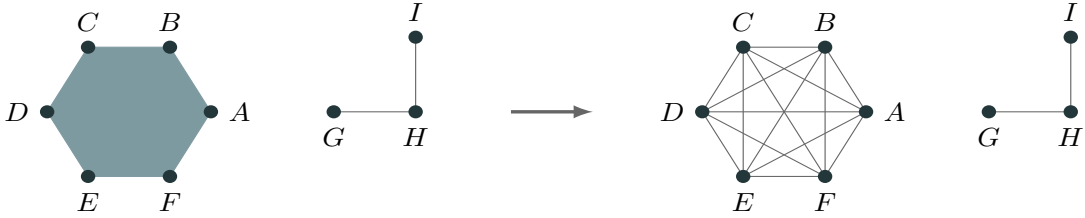
\begin{figure}
		\centering
		
		\resizebox{1\textwidth}{.2\textwidth}{
		\input{tikz_fig}
		}
		
		\caption{A synthetic coauthorship network with $n = 9$ nodes. 
		On the left, the network represented as a hypergraph with $3$ hyperedges. 
		On the right, the projected graph with $17$ dyadic edges.  
		}.  \label{fig:toy}
	\end{figure}

  \subsection*{Random Hypergraphs}
    Considerations such as these motivate the development of dedicated null models for polyadic data.
    Such models enable the analyst to delay or omit dyadic projection when conducting null-hypothesis testing.
    We now make a brief survey of efforts to define configuration-type models for polyadic data.
    Hypergraphs provide the most general context for such models.  
    Hypergraphs are straightforward generalizations of graphs in which each edge is permitted to have an arbitrary number of nodes. 
    \begin{dfn}[Hypergraph] \label{def:hypergraphs}
      A \emph{hypergraph} $H = (V,E)$ consists of a node set $V$ and an edge set $E = \{\Delta_j\}_{j = 1}^m$ which is a multiset of multisets of $V$.
      Each subset is called a \emph{hyperedge}, \emph{edge}, or, in some contexts, a \emph{simplex}. 
      Two hyperedges are \emph{parallel} if they are equal as multisets. 
      A hyperedge is \emph{degenerate} if it contains two copies of the same node. 
    \end{dfn}
    Degenerate hyperedges generalize the notion of self-loops in dyadic graphs. 
    We denote by $\mathcal{H}^\lcirclearrowright$ the set of all hypergraphs and by $\mathcal{H}$ the set of all hypergraphs without degenerate edges. 
    As before, parallel edges are permitted in $\mathcal{H}$. 
    We continue to let $n = \abs{V}$ and $m = \abs{E}$. 
    
    Extant literature provides several approaches to defining null distributions on hypergraphs.
    One of the earliest approaches \cite{Newman2001} takes a somewhat indirect route through bipartite graphs.
    A bipartite graph contains nodes of two classes, with connections permitted only between nodes of differing classes.
    To construct a bipartite graph $B$ from a hypergraph $H$, one can construct a layer of nodes in $B$ corresponding to the nodes $V$ of $H$, and a second layer in $B$ corresponding to the edges $E$ of $H$.
    A node $v$ is linked to an edge-node $e$ iff $v\in e$ in $H$.
    We can now apply dyadic configuration models to the randomize $B$, before recovering a hypergraph by projecting $B$ onto its node layer.
  	This approach is natural and convenient, but is only able to gracefully recover a generalization of the stub-labeled configuration model to hypergraphs.
    Generalizing the vertex-labeled model requires more complex tools which are not gracefully expressed in the bipartite formalism. 
  	We go into greater detail on this connection when discussing sampling methods in \Cref{sec:sampling}. 

    A more direct approach is to define a null distribution directly over $\mathcal{H}$.
    In \cite{Ghoshal2009}, the authors define an analog of the stub-labeled configuration model over the set of hypergraphs in which all edges have three nodes,  in the service of studying a tripartite tagging network on an online platform.
    Somewhat more general models have been formulated for the purposes of community-detection in hypergraphs via modularity maximization, which requires the specification of a suitable null.
    In \cite{Kumar2018}, the authors develop a degree-preserving randomization via a ``corrected adjacency matrix,'' which may then be used for modularity maximization on the projected dyadic graph.
    In \cite{Kami2018}, the authors explicitly generalize the model of Chung and Lu \cite{Chung2002}, which preserves degrees in expectation, to non-uniform hypergraphs. 
    
    One subspecies of hypergraph has received additional attention.  
    A \emph{simplicial complex} is a hypergraph with additional structure imposed by a subset-inclusion relation: if $\Delta \in E$, then $\Gamma\in E$ for all $\Gamma \subseteq \Delta$. 
    Simplicial complexes are attractive tools in studying topological aspects of discrete data \cite{Carlsson2009a}, since the inclusion condition enables often-dramatic data compression while preserving topological features of interest.
    Configuration models of simplicial complexes provide one route for conducting null hypothesis tests of such features. 
    The model of \cite{Courtney2016} achieves analytic tractability by restricting to simplicial complexes with maximal hyperedges of uniform dimension.
    The authors \cite{Young2017} allow heterogeneous dimensions but sacrifice analytic tractability, instead applying Markov Chain Monte Carlo to sample from the space.
    In applying any of these models, it is important to remember that subset-inclusion is strong property suited only to certain data-scientific  contexts.
    Particular problems arise when edges possess the semantics of interaction, such as in collaboration networks. 
    Suppose that authors $A$, $B$, and $C$ jointly coauthor a paper. 
    Using hypergraphs, we would represent this collaboration via an edge $(A,B,C)$. 
    In the setting of simplicial complexes, on the other hand, subset inclusion would also require us to include the edges $(A,B)$, $(B,C)$, $(A,C)$, $(A)$, $(B)$, and $(C)$. 
    This may be undesirable, since we are not guaranteed that $B$ and $C$, say, wrote a two-author paper. 
    While simplicial complex modeling may be useful in carefully-selected application areas, in other cases we may require more flexible configuration models defined on more general spaces of polyadic data structures. 
    We now formulate two such models. 

\section{Two Hypergraph Configuration Models} \label{sec:model_def}

	We now construct two configuration models for general hypergraphs. 
	Our models generalize the stub- and vertex-labeled dyadic configuration models described in the previous section \cite{Fosdick2018}. 
	
  We use Greek letters to denote random edges of $H$, and English letters to denote nonrandom tuples of nodes. 
	For example, the statement $\Delta = R$ describes the event that a random edge $\Delta$ has fixed location $R = (u_1,u_2,u_3,\ldots)$. 
	Let $\binom{R}{\ell}$ denote the set of all subsets of $R$ of size $\ell$. 
	Let $\mathbbm{I}$ give the indicator function of its argument. 
	We define the \emph{degree sequence} $\D \in \mathbb{Z}^n_+$ and \emph{dimension sequence} $\K \in \mathbb{Z}^m_+$ of a hypergraph $H$ componentwise by 
	\begin{align*}
	 	d_v = \sum_{e \in E}\mathbbm{I}(v \in e)\quad \text{and} \quad k_e = \sum_{v \in V}\mathbbm{I}(v \in e)\;.
	\end{align*}
	Let $\mathcal{H}_{\D, \K}^{\lcirclearrowright}$ and $\mathcal{H}_{\D, \K}$ denote the sets of hypergraphs with the specified degree and edge dimension sequences with and without degenerate hyperedges, respectively. 
  We say that the sequences $\D$ and $\K$ are \emph{configurable} of $\mathcal{H}_{\D,\K}\neq \emptyset$.   
	\begin{dfn}
		The \emph{vertex-labeled configuration model} $\eta_{\D,\K}$ is the uniform distribution on $\mathcal{H}_{\D,\K}$. 
	\end{dfn}
  The stub-labeled configuration model is defined similarly as in the dyadic case. 
  The map $g$ extends naturally to the space of hypergraphs. 
	\begin{dfn}
		Let 
		\begin{align*}
	      \Sigma = \biguplus_{v \in N}\left\{ v_1,\ldots,v_{d_v} \right\}
	    \end{align*}
	    be a multiset of \emph{stubs}. 
	    For each $v$, $d_v$ copies of $v$ appear in $\Sigma$. 
	    A \emph{stub-labeled hypergraph} $S$ has as its edge set $E$  a partition of $\Sigma$ in which each edge contains at most one stub for each node. 
	    Let $\mathcal{S}_{\D, \K}$ denote the set of all stub-labeled hypergraphs with given degree and dimension sequence, and let $\lambda_{\D, \K}$ be the uniform distribution on $\mathcal{S}_{\D, \K}$. 
      Then, the \emph{stub-labeled configuration model} is defined by $\mu_{\D, \K} = \lambda_{\D,\K} \circ g^{-1}$.  
	\end{dfn}

	We have now defined two hypergraph configuration models, generalizing the vertex- and stub-labeled models of \cite{Fosdick2018}.  
	The vertex-labeled configuration model is the entropy-maximizing distribution on $\mathcal{H}_{\D, \K}$ in the case that the identities of stubs are not meaningful, while the stub-labeled configuration model is the entropy-maximizing distribution when these identities are meaningful. 
  The same considerations discussed in \cite{Fosdick2018} (and briefly in the previous section) apply to the question of when to apply which null model. 

\section{Sampling} \label{sec:sampling}
  
  Stub-matching is a classical method for sampling from the stub-labeled dyadic configuration model \cite{Bollobas1980}, and extends naturally to the case of random hypergraphs. 
  Pseudcode for sampling from $\mu_{\D,\K}$ via stub-matching is provided by
  \Cref{alg:stub_matching}. 

  \begin{algorithm2e}[H]
        \DontPrintSemicolon
    \KwIn{Configurable $\D \in \mathbb{Z}_+^{n}$ and $\K \in \mathbb{Z}_+^{m}$.}
    \textbf{Initialization:}
    $j \gets 1,\; S \gets \emptyset,\;\Sigma  \gets \biguplus_{v \in V}\left\{ v_1,\ldots,v_{d_v} \right\}$\;
    \For{$j = 1,\ldots,m$}{
        $R \gets \binom{W}{k_j}$\;
      $W \gets W \setminus R$\;
      $S \gets S \cup \{R\}$\; 
     }
     \KwOut{$S$}
     \caption{Hypergraph Stub-Matching} \label{alg:stub_matching}
  \end{algorithm2e}
  Since any stub-labeled graph $S$ is as likely as any other \Cref{alg:stub_matching}, the output, conditioned on nondegeneracy, is distributed according to $\mu_{\D,\K}$. 
  As in the dyadic setting, there is nonzero probability for the output of stub-matching to produce a degenerate hypergraph. 
  This probability will generally be large in the presence of highly heterogeneous node degrees -- a common phenomenon in empirical data.  
  Many iterations of \Cref{alg:stub_matching} may therefore be necessary in order to generate a single valid sample. 
  Because of this, pure stub-matching is often not a practical method for generating random hypergraphs. 
  That said, the stub-matching algorithm is often useful in proofs involving $\mu_{\D,\K}$.  

	For practical sampling, we consider a Markov Chain Monte Carlo approach, in which we use successive, small alterations to the edge-set $E$ in order to systematically explore the space $\mathcal{H}_{\D,\K}$. 
	Our approach fits within a large class of edge-swap Markov chains \cite{Fosdick2018,verhelst2008efficient,artzy2005generating,viger2005efficient,mckay1990uniform,jerrum1990fast,carstens2015proof,strona2014fast,blitzstein2011sequential,del2010efficient,amanatidis2015graphic,greenhill2014switch,kannan1999simple} used for sampling a wide variety of random structures. 

	\begin{dfn}[Pairwise Reshuffle]
		Let $S \in \mathcal{S}_{\D,\K}$, and $\Delta,\Gamma \in S$.
		A \emph{pairwise reshuffle} $b(\Delta, \Gamma|S)$ of $\Delta$ and $\Gamma$ is a sample from the conditional distribution $\mu(\cdot|E\setminus \{\Delta, \Gamma\})$.
		Depending on context, we will regard a pairwise reshuffle as either a random map on stub-labeled hypergraphs or on pairs of hyperedges.
	\end{dfn}

	\begin{lm} \label{cor:reshuffle_properties}
      Let $S \in \mathcal{S}$.
      Let $b(\Delta, \Gamma|H) = (\Delta', \Gamma')$ be a pairwise reshuffle which results in $S' \in \mathcal{S}$.
      Then,
      \begin{enumerate}
        \item The degree and dimension sequences are preserved: $\mathrm{deg}(S) = \mathrm{deg}(S')$ and $\mathrm{dim}(S) = \mathrm{dim}(S')$.
        \item We have $\Delta' \cap \Gamma' = \Delta \cap \Gamma$.
        \item Any given realization of $b$ occurs with probability 
        \begin{align} \label{eq:q_S}
            q_\mu(\Delta, \Gamma) = 2^{-\abs{\Delta \cap \Gamma}}\binom{\abs{\Delta}+\abs{\Gamma}-2\abs{\Delta\cap\Gamma}}{\abs{\Delta}-\abs{\Delta\cap\Gamma}}^{-1}.
        \end{align}
      \end{enumerate}
    \end{lm}

    \begin{proof}
      A pairwise reshuffle may be performed via the following sequence, which is an alternative description of the final two iterations (conditioning on nondegeneracy).  
      \begin{enumerate}
        \item Delete $\Delta$ and $\Gamma$ from $E$.
        \item Construct $\Delta'$ and $\Gamma'$ as (initially empty) node sets.
        \item For each node $v \in \Delta \cap \Gamma$, add a $v$-stub to both $\Delta'$ and $\Gamma'$.
        \item From the remaining stubs, select $\abs{\Delta\setminus\Gamma}$ stubs u.a.r. and add them to $\Delta'$.
        Add the remainder to $\Gamma'$.
        \item Add $\Delta'$ and $\Gamma'$ to $E$.
      \end{enumerate}
      Each node begins with the same number of edges as it started, so degrees are preserved.
      Next, by construction, $\abs{\Delta'} = \abs{\Delta \cap \Gamma} + \abs{\Delta - \Gamma} = \abs{\Delta}$, and similarly $\abs{\Gamma'} = \abs{\Gamma}$. 
      The edge dimension sequence is thus also preserved.
      
      Finally, by construction, step 2 above preserves the intersection $\Delta \cap \Gamma$.
      There are $2^{\abs{\Delta \cap \Gamma}}$ ways to assign stubs to this intersection. 
      There are a total of $\abs{\Delta} + \abs{\Gamma} - 2\abs{\Delta \cap \Gamma}$ remaining stubs, and of these one must choose $\abs{\Delta}-\abs{\Delta\cap\Gamma}$ to be placed in $\Delta$.
      We infer that any given pairwise reshuffle is realized with probability given by \Cref{eq:q_S}, as was to be shown. 
    \end{proof}

    We now define a transition kernel of a first-order Markov chain on the space $\mathcal{S_{\D,\K}}$.
    Write $S \sim_{\Delta, \Gamma} S'$ if there exists a pairwise shuffle $b$ such that $b(\Delta, \Gamma|S) = S'$.
    Note that, since each element of each edge has a distinct label in $\mathcal{S}_{\D,\K}$, for any $S$ and $S'$ there is at most one pair $(\Delta,\Gamma)$ such that $S\sim_{\Delta,\Gamma}S'$.
    If no such pair exists, we write $S \not\sim S'$. 
    Then, let
    \begin{align}
      \tilde{p}_\mu(S'|S) =
      \begin{cases}
          \binom{m}{2}^{-1}q_\mu(\Delta, \Gamma) &\quad
           S \sim_{\Delta, \Gamma} S' \\
        0 &\quad S \not\sim S'\;,
      \end{cases} \label{eq:stub_kernel}
    \end{align}
    where $q_\mu(\Delta,\Gamma)$ is the number of distinct possible shuffles realizable from $\Delta$ and $\Gamma$; an explicit expression is given in the SI.  
    To sample from $\tilde{p}_\mu(\cdot|S)$, it suffices to sample two uniformly random edges from $E$ and perform a reshuffle.
    The prefactor $\binom{m}{2}^{-1}$ gives the probability that any two given edges are chosen. 

    Th sequence $\{S_t\}$ is Markovian by construction.
    The following lemma and its corollary ensure that the sequence $\{H_t\} = \{g(S_t)\}$ is also a Markov chain. 
    \begin{lm} \label{lm:labeling}
        Let $H,H' \in \mathcal{H}$. 
        Suppose that $S_1,S_2 \in g^{-1}(H)$ and $S_1',S_2' \in g^{-1}(H')$. 
        Then, $\tilde{p}_{\mu}(S_1'|S_1) = \tilde{p}_{\mu}(S_2'|S_2)$\;.
    \end{lm}
    \begin{proof}
        The objects $S_1$ and $S_2$ may each be considered arbitrary stub-labelings of part-edges in $H$. 
        Similarly, $S'_1$ and $S'_2$ are each arbitrary labelings of part-edges in $H'$. 
        However, by \Cref{eq:stub_kernel}, $\tilde{p}_\mu(\cdot|S)$ depends only on the sizes of edges and their intersections in $H$, not their labels. 
    \end{proof}

    \begin{cor} \label{cor:stub_equilibrium}
        The process $\{H_t\} = \{g(S_t)\}$ on $\mathcal{H}_{\D,\K}$ is a Markov chain. 
    \end{cor}
    \begin{proof}
        Markovianity of $\{H_t\}$  follows from \Cref{lm:labeling}. 
        Indeed, we can construct $H_t$ mechanistically from $H_{t-1}$ by choosing $S_{t-1}\in g^{-1}(H_{t-1})$, setting $S_t \sim \tilde{p}_\mu(\cdot|S_{t-1})$, and then letting $H_t = g(S_{t-1})$. 
        \Cref{lm:labeling} ensures that the distribution of $H_t$ depends only  on $H_{t-1}$, and not on the choices of $S_{t-1}$ and $S_{t}$.
	\end{proof}

    \begin{thm}\label{thm:stub_MH}
      The Markov chain $\{S_t\}$ on $\mathcal{S}_{\D,\K}$ defined by the kernel $\tilde{p}_\mu$ is  irreducible and reversible with respect to $\lambda_{\D,\K}$, the uniform distribution on $\mathcal{S}_{\D,\K}$.
      If in addition at least two entries of $\mathbf{k}$ are two or larger, $\{S_t\}$ is also aperiodic. 
      In this case, $\lambda_{\D,\K}$ is the equilibrium distribution of $\{S_t\}$. 
      Furthermore, $\mu_{\D,\K}$ is the equilibrium distribution of the process $\{H_t\} = \{g(S_t)\}$. 
    \end{thm}
    \begin{proof}
    We will first show reversibility with respect to $\lambda_{\D,\K}$.
      Fix $S$. 
      Let $S\sim_{\Delta,\Gamma}S'$ and $S' \sim_{\Delta', \Gamma'}S$.  
      In this case, we have $\Delta', \Gamma' = b(\Delta,\Gamma|S)$. 
      Then, 
      \begin{align*}
        \tilde{p}_\mu(S'|S) &= \binom{m}{2}^{-1}q_\mu(\Delta, \Gamma)   
        = \binom{m}{2}^{-1}q_\mu(\Delta', \Gamma')  
        = \tilde{p}_\mu(S|S')\;, 
      \end{align*}
      as required. 
      The second equality follows from \Cref{cor:reshuffle_properties}, since $q(\Delta,\Gamma)$ depends only on $\abs{\Delta}$, $\abs{\Gamma}$, and $\abs{\Delta \cap \Gamma}$.

      Our proof approach for irreducibility generalizes that of \cite{Fosdick2018}.
      We need to construct a path of nonzero probability between two arbitrary elements $S_1$ and $S_2$ of $\mathcal{S}$.
      Let $E_1$ and $E_2$ be the edge-sets of $S_1$ and $S_2$, respectively.
      We first describe a procedure for generating a new stub-labeled hypergraph $S_3$ such that $\abs{E_2\setminus E_3} < \abs{E_2\setminus E_1}$.
      Since $E_1 \neq E_2$ and $\abs{E_1} = \abs{E_2}$, we may pick $\Delta = \{\delta_1,\ldots,\delta_\ell\} \in E_2 \setminus E_1$.
      Note that, since $\Delta \notin E_1$ and the edge dimension sequences must agree, there exists an edge $\Psi \in E_1 \setminus E_2$ such that $\abs{\Psi} = \abs{\Delta} = \ell$.
      Now, for each $i$, since $\Delta \notin E_1$, $\delta_i$ belongs to a different edge (call it $\Gamma_i$) in $E_1$. 
      Note that we may have $\Gamma_i = \Gamma_{i'}$ in case $\delta_i$ and $\delta_{i'}$ belong to the same hyperedge in $E_1$. 
      Suppose we have $j \leq \ell$ such edges. 
      Since $\delta_i$ is a stub, $\delta_i$ can belong to only one edge in each hypergraph, and therefore $\Gamma_k \notin E_2$ for each $k = 1,\ldots,j$.
      For each $k  = 1,\ldots,j$, let $\left(\Psi_k, \Gamma_k'\right) = b_k(\Psi_{k-1},\Gamma_k)$, where $b_{k}$ assigns all elements of the set $\Delta \cap \left(\Psi_{k-1} \cup \Gamma_{k-1}\right)$ to $\Psi_k$ and uniformly distributes the remainder.
      Since $\Delta \subseteq \left(\bigcup_{k = 1}^{j} \Gamma_k\right)$ by construction, by the end of this procedure we have $\Psi_{j} = \Delta$.
      Call the resulting stub-labeled hypergraph $S_3$ with edge set $E_3$.
      Since we have only modified the edges $\{\Gamma_k\}$ and $\Psi$, which are elements of $E_1\setminus E_2$, we have not added any edges to the set $E_1\setminus E_2$, but we have removed one, namely $\Psi$.
      We therefore have $\abs{E_2\setminus E_3} < \abs{E_2\setminus E_1}$, as desired.
      Applying this procedure inductively, we obtain a path of nonzero probability between $S_1$ and $S_2$, proving irreducibility. 
      
      To prove aperiodicity, we will construct supported cycles of length $2$ and $3$ in $\mathcal{S}$. 
      Since the lengths of these cycles are relatively prime, aperiodicity will follow. 
      To construct a cycle of length 2, pick two edges $\Delta$ and $\Gamma$ and any valid reshuffle $b:(\Delta,\Gamma) \mapsto (\Delta', \Gamma')$. 
      Then, $b^{-1}:(\Delta', \Gamma') \mapsto (\Delta,\Gamma)$ is also a valid reshuffle, and the sequence $(b, b^{-1})$ of transitions constitutes a supported cycle through $\mathcal{S}$ of length $2$. 
      To construct a cycle of length 3, choose two edges $\Delta$ and $\Gamma$ which each contain two or more nodes, writing $\Delta = \{\delta_1,\delta_2,\ldots\}$ and $\Gamma = \{\gamma_1, \gamma_2,\ldots\}$. 
      This is always possible by hypothesis. 
      Then, the following sequence of pairwise reshuffles constitutes a cycle of length 3: 
      \begin{align*}
        \{\delta_1,\delta_2,\ldots\}, \{\gamma_1, \gamma_2,\ldots\} &\mapsto \{\gamma_1,\delta_2,\ldots\}, \{\delta_1, \gamma_2,\ldots\} \\
        &\mapsto \{\gamma_2,\delta_2,\ldots\}, \{\delta_1, \gamma_1,\ldots\} \\
        &\mapsto \{\delta_1,\delta_2,\ldots\}, \{\gamma_1, \gamma_2,\ldots\}\;. 
      \end{align*}
      We have shown reversibility, irreducibility, and aperiodicity, completing the proof. 
      \end{proof}
   	A small modification enables sampling from the vertex-labeled model $\eta_{\D,\K}$.
    Let $m_{\Delta}$ give the number of edges parallel to edge $\Delta$ in hypergraph $H$, including $\Delta$ itself. 
    Define 
    \begin{align}
        a_\eta(S'|S) = 
        \begin{cases}
            \frac{2^{\abs{\Delta\cap\Gamma}}}{m_\Delta m_\Gamma} &\quad S \sim_{\Delta,\Gamma} S'\\ 
             0 &\quad \mathrm{otherwise.}
        \end{cases}
    \end{align}

    \begin{thm} \label{thm:vertex_MH}
        Let $\tilde{p}_\eta(S'|S) = a(S'|S)\tilde{p}_{\mu}(S'|S)$.
        Let $\{S_t\}$ be the Markov chain generated by $\tilde{p}_\eta$. 
        Then, the process $\{H_t\} = \{g(S_{t})\}$ is a Markov chain.
        Furthermore, $\{H_t\}$ is irreducible and reversible with respect to $\eta_{\D,\K}$. 
        If in addition $\K$ has at least two entries larger than $2$, $\{H_t\}$ is aperiodic. 
        In this case, $\eta_{\D,\K}$ is the equilibrium distribution of $\{H_t\}$. 
    \end{thm}
    \begin{proof}
        Markovianity of $\{H_t\}$ follows from the same argument as \Cref{cor:stub_equilibrium}. 
        Irreducibility and aperiodicity follow from \Cref{thm:stub_MH}, since the state space $\mathcal{H}$ is a partition of $\mathcal{S}$ into equivalence classes induced by $g$.  
        It remains to demonstrate reversibility with respect to $\eta_{\D,\K}$. 
        Let $p_\eta$ be the transition kernel of $H_t$. 
        Fix $H$ and $H'$. 
        Fix $S \in g^{-1}(H)$ and $S^* \in g^{-1}(H')$. 
        Then, we can write
        \begin{align*}
            p_\eta(H'|H) &= \sum_{S' \in g^{-1}(H')}\tilde{p}_\eta(S'|S) \\ 
            &= \sum_{S' \in g^{-1}(H')}a(S'|S)\tilde{p}_{\mu}(S'|S) \\
            &= a(S^*|S)\sum_{S' \in g^{-1}(H')}\tilde{p}_{\mu}(S'|S)\;. 
        \end{align*}
        The expressions appearing in this calculation are independent of the specific choices of $S$ or $S^*$ following the same argument as in the proof of \Cref{lm:labeling}. 
        We now evaluate the sum in the third line. 
        The summand is nonzero if and only if $S \sim_{\Delta,\Gamma} S'$, in which case its value depends only on $\abs{\Delta}$, $\abs{\Gamma}$, and $\abs{\Delta \cap \Gamma}$. 
        We therefore count terms. 
        There are $2^{\abs{\Delta \cap \Gamma}}$ ways to arrange the intersection of $\Delta$ and $\Gamma$ in $\mathcal{S}$, and $m_{\Delta}m_{\Gamma}$ ways to choose two edges parallel to $\Delta$ and $\Gamma$ to reshuffle, all of which generate a distinct element of $g^{-1}(H')$. 
        The sum therefore possesses precisely $a(S^*|S)^{-1}$ terms.
        We  find that $p_{\eta}(H'|H) = \tilde{p}_\mu(S'|S)$ for any $S \in g^{-1}(H)$ and $S' \in g^{-1}(H')$. 
        Reversibility of $p_{\eta}$ thus follows from reversibility of $\tilde{p}_\mu$.
    \end{proof}

	\Cref{alg:MCMC} supplies pseudocode for sampling from the stub- and vertex-labeled hypergraph configuration models.

    \begin{algorithm2e}[H]
       \DontPrintSemicolon
        \caption{Markov Chain Monte Carlo for hypergraph configuration models}\label{alg:MCMC}
        \KwIn{$\mathbf{d}$, $\mathbf{k}$, target distribution $\nu \in \{\mu_{\D,\K}, \eta_{\D,\K}\}$, initial hypergraph $H_0 \in \mathcal{H}_{\D,\K}$, sample interval $h \in \mathbb{Z}_+$, desired sample size $s \in \mathbb{Z}_+$.}
        
        \textbf{Initialization:} $t \gets 0$, $H \gets H_0$\;
        \For{$t =1,2,\ldots, sh$}{
            sample $(\Delta,\Gamma)$ u.a.r. from $\binom{E_t}{2}$ \;
            $H' = b(\Delta, \Gamma|H_t)$ \;
            \uIf{$\mathrm{Uniform}([0,1]) \leq a_\nu(H'|H)$}{
             $H_{t} \gets H'$ \;
            }
            \Else
            {
            $H_t \gets H_{t-1}$
            }
        }
        \KwOut{$\{H_t \text{ such that } t|h\}$}
    \end{algorithm2e}

   	\Cref{thm:stub_MH,thm:vertex_MH} constitute a guarantee that, for sufficiently large sample intervals $h$, the hypergraphs sampled from \cref{alg:MCMC} will be asymptotically i.i.d. according to the desired distribution.
    Unfortunately, we are unaware of any mixing-time bounds for this class of Markov chain. 
    It is therefore possible in principle that the scaling in the mixing time as a function of system size is extremely poor, a result suggested by work on related classes of edge-swap Markov chains \cite{greenhill2011polynomial,greenhill2014switch}. 
    Our experience indicates, however, that sampling is possible for configuration models with hundreds of thousands of edges on personal computing equipment in practical time.

    \subsection{Connections to Random Bipartite Graphs}

    	As briefly mentioned in \cref{sec:lit_review}, a hypergraph $H = (V,E)$ corresponds in a natural way to a bipartite dyadic graph $B$. 
    	The graph $B$ consists of a node set $V \cup E$. 
    	An edge $(u,e)$ exists between $u \in V$ and $e \in E$ iff $u \in e$ (in $H$). 
    	In this setting, the degree of $u$ (in $H$) is equal to its degree in $B$, and the dimension of $e$ (in $H$) is similarly equal to its degree in $B$. 
    	Let $h$ be the function that assigns to each hypergraph its associated bipartite graph. 
    	When both nodes and edges are uniquely labeled, $h$ is a bijection. 
    	It follows that a probability measure $\nu$ on the space $\mathcal{B}_{\D, \K}$ of bipartite graphs with node degrees $\D$ and $\K$ induces a probability measure $\nu \circ h^{-1}$ on  $\mathcal{H}_{\D,\K}$. 
    	Several extant papers (e.g. \cite{Newman2001,Saracco2015}) use this equivalence to construct random models of polyadic data.       
    	While it is sometimes thought that bipartite randomization supplies a complete solution to null hypergraph sampling, we show in this section that the natural scope of the bipartite method is limited to stub-labeled models. 

    	We first define a bipartite, dyadic, configuration model. 
    	We define $\nu_{\D,\K}$ to be the measure on $\mathcal{B}_{\D,\K}$ obtained by performing stub-matching with the node-set $V \cup E$, conditioned on the events that (a) all edges have the form $(u,e)$ for $u\in V$ and $e \in E$, and (b) the bipartite graph is simple, without multi-edges or self-loops. 
      Note that conditioning on the event that $B$ is simple implies that the stub-labeled and vertex-labeled models are identical in this case. 
    	The work of Kannan et al. \cite{kannan1999simple} considers the problem of sampling from $\nu_{\D,\K}$ via bipartite edge-swaps. 
    	Such a swap maps $(u,e), (v,f) \mapsto (u,f), (v,e)$. 
    	By construction, such a swap preserves $\D$ and $\K$. 
    	The authors show that a Markov chain which performs successive, random bipartite edge-swaps (while avoiding ones that would lead to a non-simple bipartite graph) is ergodic and therefore sufficient to sample from $\nu_{\D,\K}$. 
    	Such a swap, which viewed in the space $\mathcal{H}_{\D,\K}$ amounts to swapping the edge memberships of nodes $u$ and $v$. 
    	Importantly, a sequence of such switches is special case of the pairwise reshuffle Markov chain on $\mathcal{S}_{\D,\K}$. 
    	This implies following relationship: 
	    \begin{prop} \label{thm:bipartite_equivalence}
        The configuration model on simple bipartite graphs is equivalent to the stub-labeled hypergraph configuration model, in the sense that
        $\mu_{\D,\K} = \nu_{\D,\K} \circ h^{-1}$.
	    \end{prop}
	    \Cref{thm:bipartite_equivalence} makes precise the primary sense in which bipartite randomization provides an approach to random hypergraph modeling. 
	    This is a convenient result, since a dyadic edge-swap Markov chain on $B$ can be used to produce samples from $\mu_{\D,\K}$.
	    This equivalence may also be used to give alternative proofs of \Cref{thm:stub_MH}. 
	    However, as discussed in \cite{Fosdick2018}, many data sets in which we aim to apply null modeling are better represented by vertex-labeled null distributions. 
	    There is no obvious route for vertex-labeled sampling through bipartite random graphs. 
	    In particular, there is no analogue of \Cref{thm:bipartite_equivalence} for this case. 
      Thus, even though the work of \cite{kannan1999simple} treats vertex-labeled sampling from $\mathcal{B}_{\D,\K}$, this does not directly suffice for vertex-labeled sampling from $\mathcal{H}_{\D,\K}$. 
	    The reason is that sampling from the vertex-labeled measure $\eta_{\D,\K}$ requires adjusting for permutations of parallel hyperedges. 
	    When $H$ contains multiple hyperedges of dimension three or greater, it is necessary to track multiple node-edge incidence relations in order to check when hyperedges are parallel. 

	    It is possible to write down a version of \Cref{alg:MCMC} for vertex-labeled sampling in which the fundamental data structure is a bipartite graph rather than a hypergraph. 
	    However, the result would not, to the author's knowledge, correspond to any standard random bipartite graph model. 
	    Expressing both models directly on the space $\mathcal{H}_{\D,\K}$ of hypergraphs supports both conceptual clarity and a convenient formulation of MCMC for both stub- and vertex-labeled models.  
	    Incidentally, we note that this discussion constitutes another, separate setting in which adherence to dyadic methods can limit our data-analytical horizons.
      An exclusive focus on nulls realizable through bipartite methods obscures the possibility of vertex-labeled polyadic models.  

\section{Network Analysis with Random Hypergraphs} \label{sec:applications}

    We now illustrate the application of hypergraph configuration models through three simple network analyses. 
    We first study triadic closure in polyadic networks, finding that the use of polyadic nulls can generate directionally different, interpretable study findings when compared to dyadic nulls.  
    We then turn to degree-assortativity, defining and testing three distinct measures of association via polyadic data representations and randomizations.
    Finally, we study the tendency of edges to intersect on multiple vertices in the \texttt{email-Enron} data set, finding using simulation and analytical methods that large intersections occur at much higher rates than would be expected by random chance. 
    Collectively, these cases illustrate the use of polyadic methods to define and analyze richer measures of network structure, and the use of polyadic nulls in interpreting the results. 
    
    The data sets for case study were gathered, cleaned, and generously made public by the authors of \cite{Benson2018}. 
    In certain experiments, data were temporally filtered in order to reduce their size; these cases of have been explicitly noted in the text and the filtering procedure described in \Cref{sec:data_prep}.
    Importantly, in no case was the filtering operation motivated by the expense of Monte Carlo sampling; rather, the bottlenecks were standard, expensive computations such as triangle-counting in dyadic graphs. 

\subsection{Triadic Closure}

    Triadic closure refers to the phenomenon that, in many networks, if two nodes $u$ and $v$ interact with a third node $w$, then it is statistically likely that $u$ and $v$ also interact with each other.
    Studies such as \cite{Strogatz1998b, Newman2001b,Newman2001} observed triadic closure in many empirical networks, and highlighted the fact that dyadic configuration models tend to be unable to reproduce this behavior.    
    Traditionally, triadic closure is measured by a ratio of the number of triangles (closed cycles on three nodes) that are present in the graph, compared to the number of ``wedges'' (subgraphs on three nodes in which two edges are present).\footnote{Recent measures have been developed for higher-order notions of clustering on larger subgraphs; see \cite{Yin2018a}.} 
    Local and global variants of this ratio have been proposed. 
    We follow the choice of \cite{Strogatz1998b} and work with the \emph{average local clustering coefficient}. 
    Let $T_v$ denote the number of triangles incident on $v$, and $W_v$ the number of wedges. 
    Note that $W_v = \binom{d_v}{2}$. 
    The average local clustering coefficient is 
    \begin{align}
        \bar{C} = \frac{1}{\abs{N}}\sum_{v \in N}\frac{T_v}{W_v}\;. 
    \end{align}
    It is direct to show \cite{Newman2010} that, in dyadic configuration models and under mild sparsity assumptions, $\bar{C}$ decays to zero as $n$ grows large. 
    
    The average local clustering coefficient $\bar{C}$ is a natively dyadic metric, in the sense that ``wedges'' and ``triangles'' are defined explicitly in terms of 2-edges. 
    To compute $\bar{C}$ in polyadic data, it is therefore necessary to project a hypergraph down to a dyadic graph. 
    In the context of hypothesis-testing, there is some subtlety involved in the choice of when to do this. 
    One method is to project first and then randomize via a dyadic null model. 
    This is the most common historical approach, used for example in \cite{Newman2001}.
    Alternatively, one may randomize via polyadic nulls prior to projection. 
    This approach has the effect of preserving clustering induced by polyadic edges, since an edge of dimension $k$ contains $3\binom{k}{3}$ wedges and $3\binom{k}{3}$ ordered triangles. 
    
    \begin{table}
    	\centering
    	\ra{1.3}
    	\begin{tabular}{@{}llllllll@{}}
    		\multicolumn{1}{c}{}& \multicolumn{1}{c}{}&\phantom{a} & \multicolumn{2}{c}{Hypergraph} &\phantom{a}& \multicolumn{2}{c}{Projected}\\ 
    		\cmidrule{4-5}\cmidrule{7-8}
    		&$\bar{C}$& & Vertex & Stub && Vertex & Stub \\\midrule
    		\texttt{congress-bills}* & 0.608 && 0.601(1) & 0.622(2)  && 0.451(2)  & 0.611(1) \\
    		\texttt{coauth-MAG-Geology}* & 0.8200 && 0.8196(7)  & 0.8186(7)  && 0.00035(3) & 0.00035(3) \\
    		\texttt{email-Enron} & 0.658 && 0.825(3)  & 0.808(4)  && 0.638(5) & 0.797(3) \\
    		\texttt{email-Eu}* & 0.540 && 0.569(4) & 0.601(4)  && 0.398(4)  & 0.598(4) \\
    		\texttt{tags-ask-ubuntu}* & 0.571 && 0.609(4)  & 0.631(5)  && 0.183(4) & 0.499(6)  \\
    		\texttt{threads-math-sx}* & 0.293 && 0.435(3) & 0.426(3) && 0.041(1)  & 0.093(2) \\ \\ 
    	\end{tabular}
    	\caption{
    		Average local clustering coefficients for selected data sets, compared to their expectations computed under vertex- and stub-labeling of hypergraph and projected graph models. 
    		Parentheses show standard deviations in the least-significant figure under the equilibrium distribution of each null model.
    		Starred* data sets have been temporally filtered as described in \Cref{sec:data_prep}.
    		} \label{tb:clustering}
	\end{table}

    \Cref{tb:clustering} summarizes a sequence of experiments performed on two collaboration networks (top) and four communication networks (bottom).
    For each network, we  computed the observed local clustering coefficient $\bar{C}$ on the unweighted projected graph. 
    We then compared the observed value to its null distribution under four randomizations. 
    We first randomized using the vertex- and stub-labeled hypergraph configuration models, \emph{prior} to projecting and measuring $\bar{C}$.
    These results are shown in the second and third columns. 
    We then reversed the order, first computing the projected graph and randomizing via dyadic configuration models.
    The results are shown in the fourth and fifth columns. 
    
    Benchmarking against dyadic configuration models yields mixed results. 
    Vertex-labeled configuration models conclude in all cases that the observed degree of clustering is significantly higher than would be expected by random chance. 
    Stub-labeled benchmarking concludes that \texttt{congress-bills} and the two email data sets have significantly less clustering than expected, while the remainder have significantly more. 
    The stub-labeled results should be approached with caution -- for reasons discussed in detail in \cite{Fosdick2018}, the stub-labeled configuration model is a less-relevant comparison for these data sets than the vertex-labeled model. 
    
    Hypergraph randomization leads to directionally different conclusions. 
    First, the expected values of $\bar{C}$ under both hypergraph vertex- and stub-labeled nulls are much closer together than under dyadic nulls, indicating that the polyadic statistical test is much less sensitive than the dyadic test to the choice of vertex- and stub-labeling.
    Second, the vertex-labeled null separates the two collaboration networks from the four communication networks. 
    These two data sets are only slightly more clustered than expectation under the vertex-labeled model, and only \texttt{congress-bills} would be considered ``significantly more clustered'' than its expectation under most $p$-value based tests. 
    The stub-labeled model also has expectation close to the observed values, but finds \texttt{congress-bills} to be slightly-but-significantly less clustered than would be expected by chance, while the significance of \texttt{coauth-MAG-Geology} would depend strongly on the desired power of the test. 
    In contrast, the four communication networks are all significantly \emph{less} clustered than either vertex- or stub-labeled nulls would expect. 
    Not only is there no clustering beyond that implied by the edge dimensions; triadic closure even appears to be inhibited in these data sets. 

    From a purely statistical perspective, these examples highlight the importance of careful null model selection in hypothesis-testing for triadic closure. 
    More physically, use of hypergraph nulls allows us in this case to distinguish data sets by their generative mechanisms. 
    The communication networks are all less clustered than expected, while the collaboration networks are approximately as clustered as expected. 
    This result is to some extent intuitive. 
    Collaborations between many agents often have nontrivial coordination costs that scale with the number of agents involved. 
    It may be easier to assemble and coordinate a set of overlapping groups than a single large collective. 
    In such cases, one may expect to observe clustering near or above that expected at random, since overlaps between related groups would generate triangles.
    In contrast, in digital communications it is essentially effort-free to construct interactions between larger groups of agents. 
    Examples include adding an email address to the ``cc'' field or introducing participants to thread on a forum.
    In such cases, triangles composed of distinct edges are energetically unnecessary, and may reflect redundant information flow. 
    We therefore hypothesize that these systems have a tendency to absorb potential triangles into higher-dimensional interactions. 
    This results in lower levels of clustering than would be expected under polyadic nulls. 
    These considerations hint toward the importance of studying edge correlations via natively polyadic metrics as we do in \Cref{subsec:intersection}.
    
  \subsection{Degree-Assortativity} \label{subsec:assortativity}

    A network is degree-assortative when nodes of similar degrees  preferentially interact with each other. 
    Early studies found that different categories of social, biological, and technological networks display different patterns of assortative mixing by degree \cite{Newman2003b,Newman2002a,Colizza2006a}. 
    Social networks, for example, are frequently measured to be degree-assortative. 
    In this context, degree-assortativity is often taken to indicate a tendency for popular or productive agents to interact with each other. 

    We measure degree-assortativity in hypergraphs via a generalization of the standard Spearman rank assortativity coefficient to hypergraphs. 
    Importantly, there are multiple possible generalizations, each of which measures distinct structural information about degree correlations. 
    Let $E_{\geq 2} = \{\Delta \in E : \abs{\Delta} \geq 2 \}$.   
    Let $h:E_{\geq 2}\rightarrow N^2$ be a (possibly random) \emph{choice function} that assigns to each edge $\Delta$ two distinct nodes $u,v \in \Delta$. 
    Three possibilities of interest are: 
    \begin{align}
        h(\Delta) &= (u,v) \sim \mathrm{Uniform}\binom{\Delta}{2} \tag{Uniform}\\
        h(\Delta) &= (u,v) = \argmax_{w,w' \in \binom{\Delta}{2}} d_wd_{w'} \tag{Top-2} \\ 
        h(\Delta) &= (u,v) = \left(\argmax_{w \in \Delta} d_w,\; \argmin_{w \in \Delta} d_w\right) \tag{Top-Bottom}
    \end{align}
    The Uniform choice function selects two distinct nodes at random. 
    The Top-2 choice function selects the two distinct nodes in the edge with largest degree. 
    The Top-Bottom choice function selects the  nodes with largest and smallest degree. 

    Let $r:N\rightarrow \R$ be a ranking function on the node set; we will always take $r(u)$ to be the rank of node $u$ by degree in the hypergraph. 
    For fixed $h$, let $f:E\rightarrow \R^2$ be defined componentwise by $f_j(\Delta) = (r \circ h_j)(\Delta)$.
    Then, the \emph{generalized Spearman assortativity coefficient} is the empirical correlation coefficient between $f_1(\Delta)$ and $f_2(\Delta)$:
    \begin{align}
        \rho_h = \frac{\sigma^2\left(f_1(\Delta),f_2(\Delta)\right)}{\sqrt{\sigma^2\left(f_1(\Delta), f_1(\Delta)\right)\sigma^2\left(f_2(\Delta), f_2(\Delta)\right)}}, \label{eq:spearman}
    \end{align}
    where $\sigma^2(X,Y) = \bracket{XY} - \bracket{X}\bracket{Y}$ and brackets express averages over pairs of edges in $E$. 
 
    In the case of dyadic graphs, the three choice functions above are trivially identical, since there is only one way to pick two nodes from an edge of size two.  
    On polyadic data, however, the resulting Spearman coefficients capture usefully different classes of information. 
    For example, in studying coauthorship networks, they may be used to test hypotheses such as the following: 
    \begin{enumerate}
        \item \textbf{Generic Assortativity}: On a given paper, most coauthors will simultaneously be more or less prolific than average. 
        \item \textbf{Junior-Senior Assortativity}: The least profilic author on a paper will tend to be relatively more prolific if the most prolific author is relatively more prolific. 
        \item \textbf{Senior-Senior Assortativity}: The two most prolific authors on a paper will tend to be simultaneously more or less prolific than average. 
    \end{enumerate}
    While the corresponding Spearman coefficients may in general be correlated, substantial variation manifests across study data sets. 
    \Cref{fig:significance} shows measurements and significance tests for one synthetic data set and the six empirical data sets studied in the previous section.
    The synthetic data consists of five copies of the hypergraph shown in \Cref{fig:toy}. 
    For each data set, we compute the dyadic assortativity coefficient on the projected graph (first row), as well as each of the three polyadic assortativity coefficients defined above. 

    \begin{figure*}
      \includegraphics[width=\textwidth]{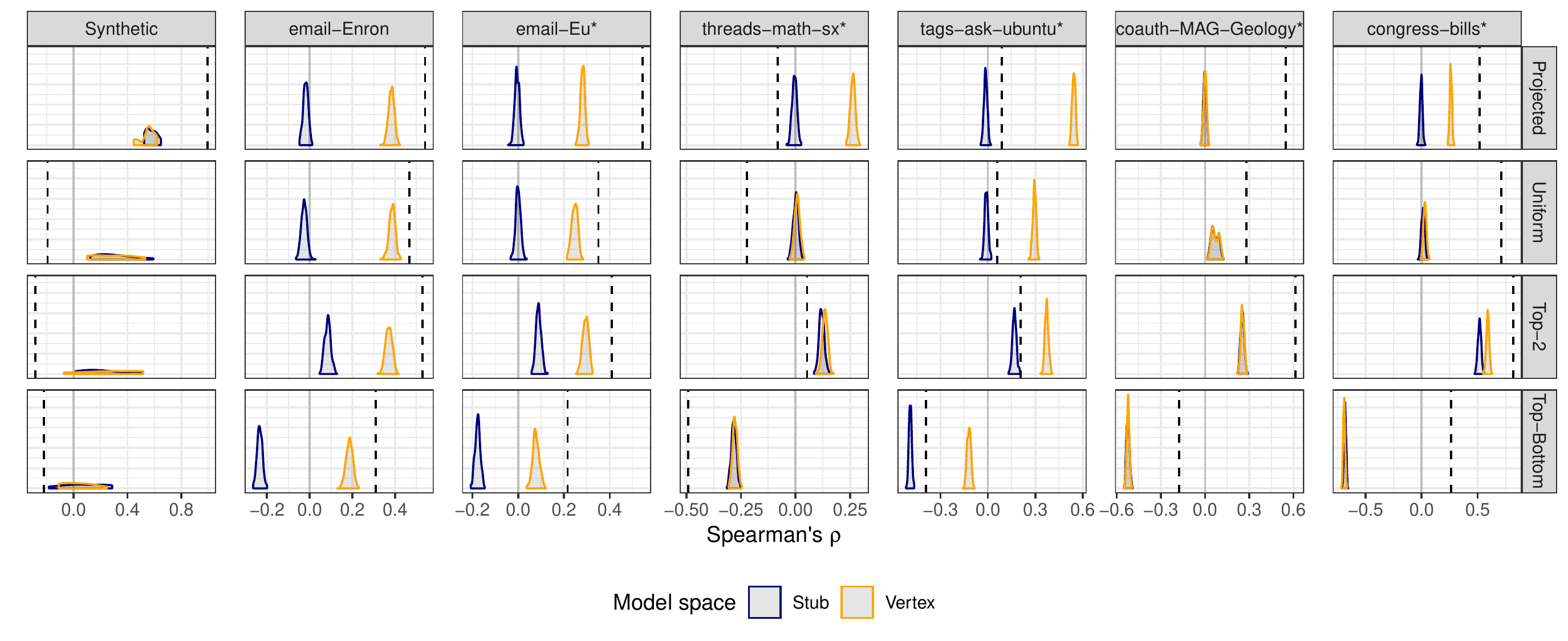}
      \caption{
        Significance tests of degree-assortativity in synthetic and empirical networks. 
        The synthetic data consists of five copies of the illustrative network shown in \Cref{fig:toy}. 
        In each figure, the dashed line gives the observed Spearman correlation, and densities give the null distributions under vertex- and stub-labeled configuration models.
        In the first row only, the hypergraph was projected down to an unweighted dyadic graph prior to randomization. 
        Starred$^*$ data sets have been temporally filtered as described in the SI. 
      } \label{fig:significance}
    \end{figure*}

    The synthetic data (first column) illustrates a stark case in which dyadic hypothesis-testing leads to a finding of statistically-significant assortativity, while polyadic hypothesis-testing finds statistically-significant \emph{dis}assortativity. 
    In each of the empirical data sets, the dyadic and polyadic tests show directional agreement. 
    However, the polyadic tests highlight several features of the data missed by the dyadic tests.   
    In the two email data sets, all four coefficients are  positive and to the right of the null distributions, though projecting (first row) increases the significance of the coefficients relative to hypergraph randomization (second row). 
    The two forum data sets (\texttt{threads-math-sx} and \texttt{tags-ask-ubuntu}) are disassortative when compared to vertex-labeled nulls. 
    The fact that \texttt{tags-ask-ubuntu} is disassortative despite a positive uniform hypergraph Spearman coefficient speaks to the importance of carefully specified null hypothesis testing. 
    Interestingly, the misspecified stub-labeled randomization would lead to the opposite finding.  
    The coauthorship network \texttt{coauth-MAG-Geology} is highly assortative in all metrics -- including the top-bottom measure, which is negative. 
    The \texttt{congress-bills} data set is also assortative in all measures. 
    Unlike the other data sets, the uniform hypergraph coefficient lies farther from the bulk of its null distribution than does the projected coefficient.
    We note in passing that, whereas the stub-labeled and vertex-labeled hypergraph distributions had similar expected rates of triadic closure \cref{tb:clustering}, their distributions of degree-assortativity coefficients vary substantially, and would in some cases lead to directionally different study conclusions. 
    
    When studying triadic closure, we saw how hypergraph null models could lead us to differently-contextualize standard graph metrics. 
    When studying assortativity, we gain even more. 
    Use of hypergraph nulls allows us to forgo the dyadic projection operation, and thereby define rich polyadic assortativity measures. 
    Hypergraph null models thus enable us to expand our network-analytic toolboxes by measuring and interpreting novel structural patterns in polyadic data. 

\subsection{Hyperedge Intersection Profiles} \label{subsec:intersection}

    Let $\Delta, \Gamma \in H$. 
    What is the size of their intersection? 
    In the case of dyadic graphs, the intersection can have size at most two, when $\Delta$ is parallel to  $\Gamma$. 
    In hypergraphs intersections of arbitrary sizes may occur. 
    The existence of large intersections in a data set may indicate the emergence of polyadic social ties between groups of agents, or interpretable event sequences such as email threads or series of related scholarly papers. 
    Several recent papers \cite{Benson2018,patania2017shape} have studied similar questions by considering the rate at which ``holes'' in the hypergraph tend to be ``filled in'' by higher-order interactions. 
    We take a simpler approach, defining a measure which is both easily computed and amenable to analytical approximation. 
    
    \begin{dfn}[Intersection Profile]
        For fixed $k,\ell \in \mathbb{Z}_+$, the  \emph{conditional intersection profile} of a hypergraph $H \in \mathcal{H}$ is the distribution
        \begin{align*}
            r_{k\ell}(j|H) = \bracket{\mathbbm{I}(\abs{\Delta \cap \Gamma} = j)}_{k\ell}\;, 
        \end{align*}
        where $\bracket{\cdot}_{k\ell}$ denotes the empirical average over all hyperedges $\Delta$ of size $k$ and $\Gamma$ of size $\ell$.
        The \emph{marginal intersection profile} is 
        \begin{align*}
            r(j|H) = \bracket{\mathbbm{I}(\abs{\Delta\cap\Gamma} = j)}\;,
        \end{align*}
        with the average taken over all pairs of distinct edges in $E$. 
    \end{dfn}
    Large values of $r_{k\ell}(j|H)$ indicates that edges of size $k$ and $\ell$ frequently have intersections of size $j$ in $H$. 
    Empirical data sets may possess complex patterns of correlation between edges of various sizes. 
    Evaluating whether an observed conditional or marginal intersection profile is noteworthy requires comparison to appropriately-chosen null models.

    \Cref{fig:intersections} demonstrates the use of hypergraph configuration models to study the intersection profile of the \texttt{email-Enron} data set.
    In \Cref{fig:intersections}(a), we compare the empirical average intersection size $\bracket{J}_{k\ell} = \sum_{j = 0}^{\infty} jr_{k\ell}(j|H)$ to its average $\bracket{\hat{J}}_{k\ell}$ under the vertex-labeled configuration model. 
    Higher values of the ratio $ \frac{\bracket{J}_{k\ell}}{\bracket{\hat{J}}_{k\ell}}$ indicate the presence of denser intersections between edges of sizes $k$ and $\ell$. 
    Notably, the empirical averages are not uniformly higher than the null model averages, even on the diagonal. 
    There is apparent block structure, indicating that edges of certain sizes tend to correlate most strongly with certain other sizes.
    Edges of dimension $3$ through $6$ tend to interact strongly with each other, as do edges of dimension $7$ and $8$. 
    However, edges in the smaller group interact more weakly with edges in the larger group than would be expected by chance. 
    Further, more detailed study may be able to shed light on the groups of agents involved in these overlapping communications. 
    
    \begin{figure*}
      \centering
      \includegraphics[width=.9\textwidth]{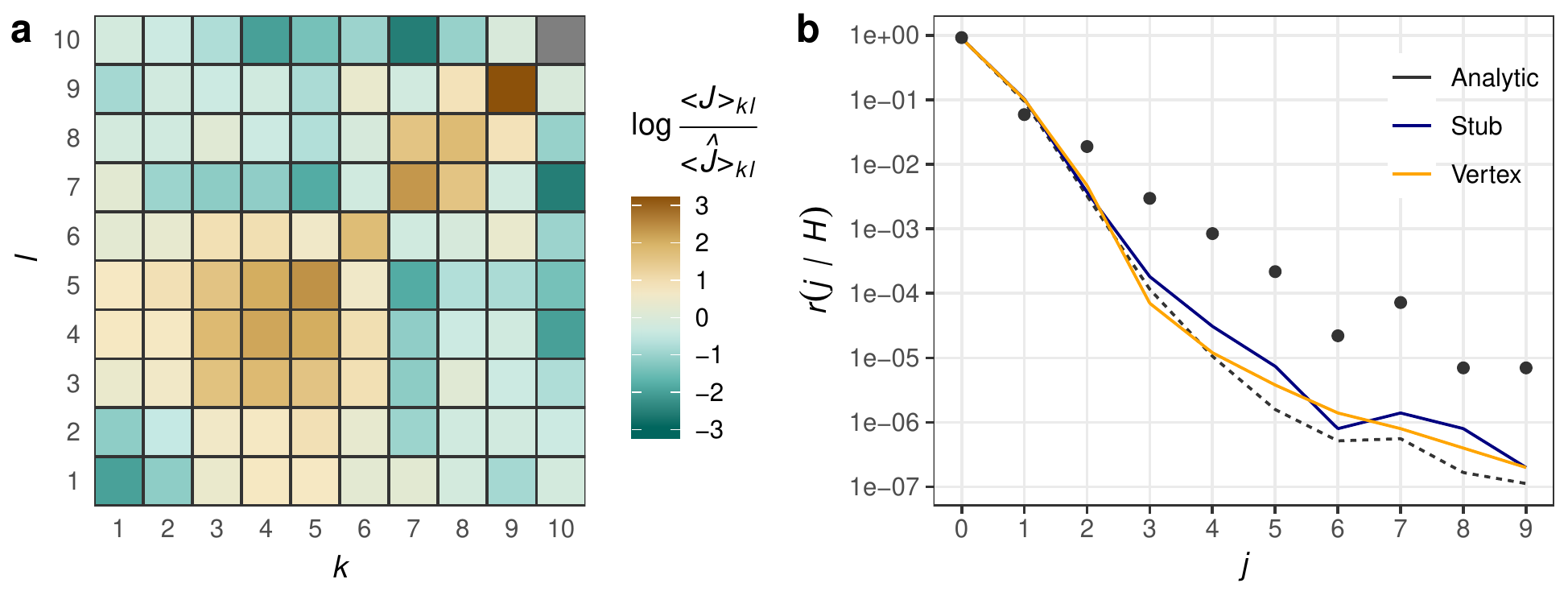}
        \caption{
            Analysis of intersection profiles in the \texttt{email-Enron} data set.  
            (\textbf{a}): The average of the intersection size normalized by the expectation $\bracket{\hat{J}}_{k\ell}$ under the vertex-labeled configuration model.  
          Positive values indicate that the data displays larger intersections than expected under the configuration model for the corresponding values of $k$ and $\ell$. 
          Colors are shown on a log scale. 
          The missing value at $(k,\ell) = (10,10)$ indicates that no nonempty intersections were observed between edges of these sizes in the Monte Carlo sampling runs. 
          (\textbf{b}): Marginal intersection profile (points) of the empirical data, compared to null distributions under the stub- and vertex-labeled configuration models. 
          The dashed gray line gives the analytic approximation of \Cref{eq:asymptotic}. 
          Note the logarithmic vertical axis. 
    } 
        \label{fig:intersections}
    \end{figure*}

    \Cref{fig:intersections}(b) gives a global view of the data using the marginal intersection profile. 
    The observed profile (points in \Cref{fig:intersections}(b)) is nearly linear on semilog axes through $j = 6$, suggesting that the decay in the intersection size is roughly exponential. 
    In order to evaluate whether this behavior indicates nonrandom clustering between edges, we again turn to hypergraph configuration models. 
    The expectation $\hat{r}(j) = \E_\nu[r(j|H)]$ of the marginal intersection profile under a configuration model $\nu \in \{\mu_{\D,\K}, \eta_{\D,\K}\}$ measures the typical behavior of a comparable random hypergraph.
    The solid lines in \Cref{fig:intersections}(b) give these expected profiles under both stub- and vertex-labeled models, which directionally agree. 
    The observed data shows fewer intersections on single vertices than would be expected by chance.
    On the other hand, for $j\geq 3$, $r(j|H)$ exceeds $\hat{r}(j)$ by an order of magnitude or more, suggesting substantial higher-order correlation in the data. 
    These results likely reflect the passing of multiple messages between the same sets of users.

    Some data sets may be too large to practically estimate $\hat{r}(j)$ by Monte Carlo methods. 
    In such cases, it is possible to approximate $\hat{r}(j)$ under the stub-labeled configuration model analytically, using the following asymptotic result. 
    \begin{thm} \label{thm:analytic}
    	Fix $\ell$, $k$, and $j$. 
    	Let $\mathbf{D}\in \mathbb{Z}_+^n$ be a vector of i.i.d. copies of positive, discrete random variable $D \in \mathbb{Z}_+$ such that $D \leq d_{\mathrm{max}}$ almost surely for some $d_{\mathrm{max}}$. 
    	Let $\mathbf{K} \in \mathbb{Z}_+^m$ be any vector of edge dimensions configurable with $\mathbf{D}$. 
    	Let $H \sim \mu_{\mathbf{D}, \mathbf{K}}$, and
    	let $\Delta$ and $\Gamma$ be uniformly random edges of $H$. 
    	Then, with probability approaching unity as $n$ grows large, 
    	\begin{align}
    		\hat{r}_{k\ell}(j) = (1 + O(n^{-1}))j! \binom{k}{j} \binom{\ell}{j}\left(\frac{1}{n}\frac{\E[D^2] - \E[D]}{\E[D]^2}\right)^j\;. \label{eq:asymptotic}
    	\end{align} 
    \end{thm}
    \begin{proof}
    	Let $\bracket{d} = \frac{1}{n}\sum_{u \in N}d_u$ denote the empirical mean degree of a given degree sequence $\D$.  
	    Assume without loss of generality that $\Delta = \{\delta_1,\ldots,\delta_k\}$ and $\Gamma = \{\gamma_1,\ldots,\gamma_\ell\}$ are the first two hyperedges formed by \Cref{alg:stub_matching}, conditioned on nondegeneracy.   
	    There are $\binom{k}{j}$ ways to choose the $j$ elements of $\Delta$ contained in $\Delta \cap \Gamma$, and similarly $\binom{\ell}{j}$ ways to choose the elements of $\Gamma$. 
	    There are then $j!$ ways to place these two sets in bijective correspondence. 
	    Define the event $A = \{\delta_h = \gamma_h,\; h = 1,\ldots,j\}$. 
	    Then, $\hat{r}_{k\ell}(j) = j!\binom{k}{j}\binom{k}{\ell}\mu_{\mathbf{D},\mathbf{K}}(A)$.
	    To compute $\mu_{\mathbf{D},\mathbf{K}}(A)$, we may explicitly enumerate
	    \begin{align*}
	        \mu_{\mathbf{D},\mathbf{K}}(A) = \sum_{u \in N} \frac{d_u}{n\bracket{d}}\frac{d_u-1}{n\bracket{d}-1} \left[\sum_{v \in N\setminus\{u\}} \frac{d_v}{n\bracket{d}-d_u}\frac{d_v-1}{n\bracket{d}-d_u-1}\left[\sum_{w \in N\setminus \{u,v\}}\cdots\right]\right]\;,
	    \end{align*}
	    with a total of $j$ sums appearing. 
	    In each summation, the first factor gives the probability that $\delta_1 = u$ and the second that $\gamma_1 = u$. 
	    Consider the innermost summation, which may be written 
	    \begin{align}
	        S_R = \sum_{z\in N\setminus R} \frac{d_z}{n\bracket{d} - \sum_{y \in R}d_y}\frac{d_z-1}{n\bracket{d}- \sum_{y \in R}d_y-1} \label{eq:to_expand}
	    \end{align}
	    for a set $R$ of size $j-1$. 
	    Since $D \leq d_{\mathrm{max}}$ a.s., we may employ Chebyshev's inequality to find that  $(n\bracket{d})^{-1}\sum_{y \in R}d_y = O(n^{-1})$ w.h.p. 
	    We may therefore w.h.p. expand both factors within
	    \Cref{eq:to_expand}, obtaining the expression
	    \begin{align*}
	        \sum_{z\in N\setminus R} 
	        \frac{d_z(d_z-1)}{n^2 \bracket{d}^2}\left(1 + O(n^{-1})\right)\;.
	    \end{align*}
	    Using Chebyshev's inequality again, we also obtain asymptotic behavior on the other expressions appearing above. $\bracket{d} = \left(1 + O\left(n^{-1}\right)\right)\E[D]$ and $\sum_{z\in N\setminus R} 
	    \frac{d_z(d_z-1)}{n} = \left(1 + O\left(n^{-1}\right)\right) (\E[D^2] - \E[D])$, both w.h.p. 
	    We have therefore shown that 
	    \begin{align}
	        S_R = \left(1 + O\left(n^{-1}\right)\right)\left(\frac{1}{n}\frac{\E[D^2] - \E[D]}{\E[D]^2}\right) \label{eq:example}
	    \end{align}
	    w.h.p. 
	    This argument may be repeated inductively for each of the remaining $j-1$ sums, each of which contributes the same factor appearing in \Cref{eq:example}, proving the theorem. 
    \end{proof}
    \Cref{fig:intersections} shows the resulting approximation for $\hat{r}_{k\ell}$ as a dashed line, finding excellent qualitative agreement. 
    This approximation may be used to study intersection profiles in data sets of arbitrary size. 
    \Cref{fig:large_intersections} shows the use of this approximation to study intersection profiles in hypergraph data sets of arbitrary size. 
   	The top set of panels shows four data sets in which the approximate null intersection profile consistently underestimates the rates of large intersections by several orders of magnitude, clearly indicating the presence of correlation structure over and above what would be expected under hypergraph randomization. 
   	The lower panels show four additional data sets in which the approximate null profile more closely-approximates the observed data. 	

    \begin{figure*}
        \includegraphics[width=\textwidth]{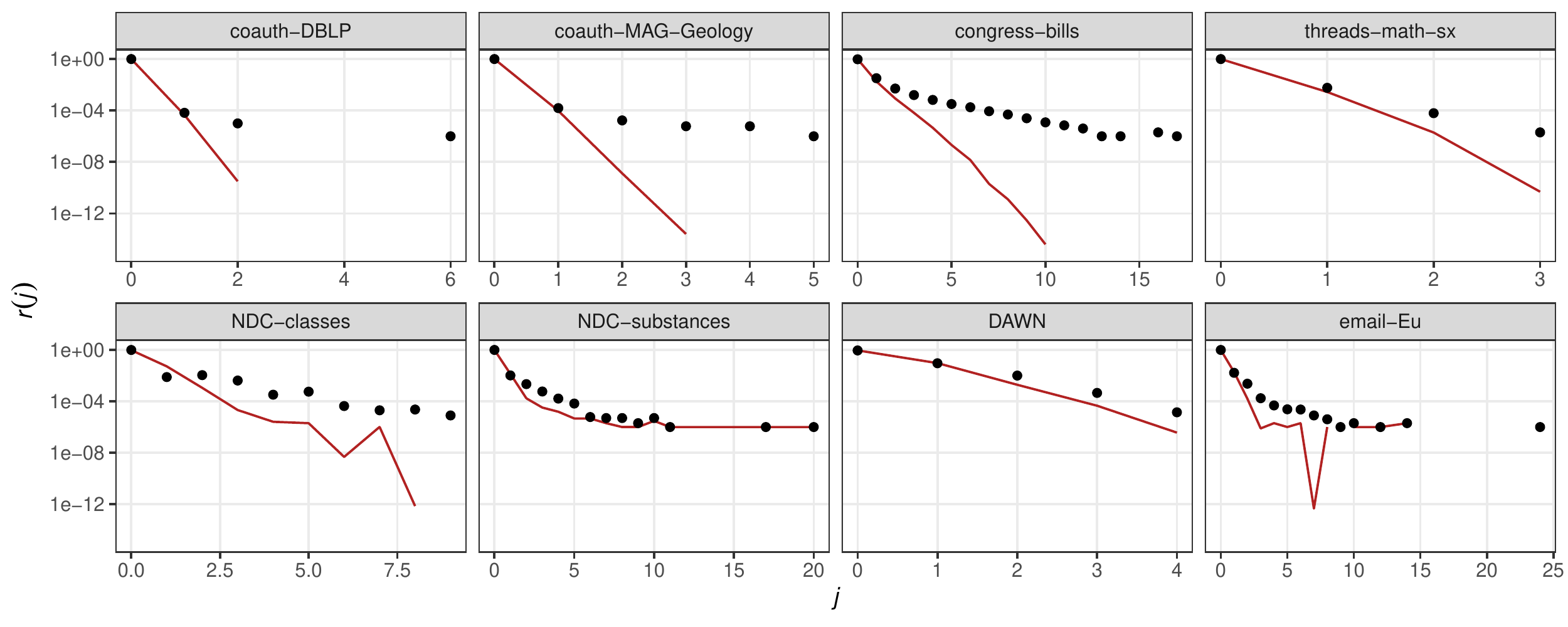}
        \caption{Points give the observed intersection profile for six large polyadic data sets. 
        The solid line gives the null intersection profile of \Cref{thm:analytic}.
        In this visualization, full data sets were used -- the temporal filtering described in \Cref{sec:data_prep} was not performed.} \label{fig:large_intersections}
    \end{figure*}

\section{Discussion}\label{sec:conclusion}

    Configuration models of random hypergraphs preserve the first moments of the data --  the degree and edge-dimension sequences -- while remaining maximally ignorant about additional data structure.
    These models extend the widely-used configuration models for dyadic graphs, and serve as natural null models for polyadic network data analysis.
    We have demonstrated how to define, sample from, analyze, and apply these models.
    We have seen that the  choice between nulls can greatly impact the directional findings of studies of empirical polyadic data. 
    The analyst faced with such a choice must therefore carefully consider whether dyadic simplification will lead to data representations and null spaces that are relevant for their application area. 
    Second, employing polyadic nulls often allows the analyst to define novel measures that can illuminate higher-order structure in data. 
    We have illustrated this with extended assortativity measures and intersection profiles, but many more extensions are possible.  
    We hope that the introduction of statistically-grounded hypergraph nulls will encourage analysts to design, measure, and carefully interpret many novel measures of polyadic network structure. 

    There are several directions of future work on configuration models of random hypergraphs. 
    Beginning with theory, many classical asymptotic results on dyadic configuration models invite generalization. 
    These include probabilistic characterization of component sizes; cycles and parallel edges; and the diameter of the connected component in various regimes. 
    We also highlight two applications of potential interest. 
    The first is motif analysis. 
    A network motif is a subgraph that appears with higher-than-expected frequency in a given network \cite{milo2002network}, relative to a given null model. 
    Considering the explicit dependence of this definition on the null, we conjecture that motif-discovery algorithms based on polyadic nulls may highlight importantly distinct structure when compared to dyadic nulls. 
    A second promising application is in hypergraph clustering and community detection. 
    A recent paper \cite{Kami2018} offers a definition of modularity --- a common quality function for network partitioning --- based on a polyadic generalization of the Chung-Lu model \cite{Chung2002}. 
    In this case, the modularity of a given partition may be computed analytically. 
    The same calculations used to prove \cref{thm:analytic} can also be used to show that the stub-labeled configuration model will give an asymptotically equivalent expression. 
    However, for the large class of data sets more appropriately modeled by vertex-labeled nulls, other methods may be necessary.
    We anticipate that pursuing these tasks will pose interesting theoretical and computational challenges. 

\subsection*{Funding}

	This work was supported by the National Science Foundation Graduate Research Fellowship under award number 1122374.

\subsection*{Acknowledgments}

	I am grateful to Patrick Jaillet for helpful discussions from which this work benefited substantially. 

\subsection*{Software}

	A \texttt{hypergraph} class, written in Python 3.5, is available at 
	\begin{align*}
		\text{\texttt{https://github.com/PhilChodrow/hypergraph}.}
	\end{align*}
	This class includes implementations of Monte Carlo sampling for both stub- and vertex-labeled configuration models. 

\subsection*{Data}

	The data sets used in this paper were prepared by the authors of \cite{Benson2018} and accessed from \texttt{https://www.cs.cornell.edu/~arb/data/}. 

\appendix

\section{Data Preparation} \label{sec:data_prep}
    The data sets used in this paper were prepared by the authors of \cite{Benson2018} and accessed from \texttt{https://www.cs.cornell.edu/~arb/data/}. 
    Some data sets have been filtered to exclude edges prior to a temporal threshold $\tau$ in order to promote practical compute times on triangle counting and mixing of vertex-labeled models in projected graph spaces. 
    Notably, in no cases was sampling from hypergraph configuration models the computational bottleneck. 
    Thresholds were chosen to construct data with edge sets of approximate size $m\approx 10^4$, but are otherwise arbitrary. 
    Temporal data subsets were used in the generation of \Cref{tb:clustering,fig:significance}.
    \Cref{tb:data} gives the node and edge counts of both the original data and the data after temporal subsetting when applicable. 
    
	\begin{table}[h]
		\centering
		\ra{1.3}
		\begin{tabular}{@{}lrrrrrrrr@{}}
		\multicolumn{1}{c}{}& \multicolumn{1}{c}{}& \multicolumn{2}{c}{Original} &\phantom{a}&\phantom{a} & \phantom{a} &\multicolumn{2}{c}{Filtered}\\ 
		\cmidrule{3-4}\cmidrule{8-9}
		& & $n$ & $m$ &&$\tau$&& $n$ & $m$ \\\midrule
		\texttt{email-Enron} && $143$ & $10,886$ &&$-$&&$-$ & $-$  \\ 
		\texttt{email-Eu} && $1,006$ & $235,264$ &&$1.105 \times 10^{9}$&& $817$ & $32,117$  \\ 
		\texttt{congress-bills} \cite{Fowler-2006-connecting,Fowler2006}&  & $1,719$  &$260,852$ && $7.315 \times 10^5$&& 537 & 6,661  \\ 
		\texttt{coauth-MAG-Geology} \cite{Sinha-2015-MAG}&& $1,261,130$ &$1,591,167$&&$2017$&& $73,436$ & $23,434$ \\
		\texttt{threads-math-sx} && $201,864$ & $719,793$ &&$2.19 \times 10^{12}$&& $11,880$ & $22,786$ \\ 
		\texttt{tags-ask-ubuntu} && $200,975$ & $192,948$ &&$2.6 \times 10^{12}$ && $2,120$ & $19,338$ \\ \\ 
		\end{tabular}
		\caption{Summary of data preparation. 
		When $\tau$ is given, the filtered data set consists in all edges that occurred after time $\tau$.} \label{tb:data}
	\end{table}



\nocite{Fowler-2006-connecting,Fowler2006,Sinha-2015-MAG}

\bibliographystyle{imaiai} 
\bibliography{references.bib}

\end{document}

%% file: tikz_fig.tex
		
\begin{tikzpicture}
	\coordinate (A) at (1,0);
	\coordinate (B) at (.5,0.86602540378);
	\coordinate (C) at (-.5,0.86602540378);
	\coordinate (D) at (-1,0);
	\coordinate (E) at (-.5,-0.86602540378);
	\coordinate (F) at (.5,-0.86602540378);

	\coordinate (G) at (2.5, 0);
	\coordinate (H) at (3.5, 0);
	\coordinate (I) at (3.5, 1);
	
	\filldraw[draw=light_teal, fill=light_teal] (A) -- (B) -- (C)  -- (D) -- (E) -- (F) -- cycle;
	
	\draw[draw=edge_gray] (G) -- (H);
	\draw[draw=edge_gray] (H) -- (I);

	\node[circle, fill=metro_teal, minimum size=5pt, inner sep=0pt] (a) at (A) [label=right:$A$]{};
	\node[circle, fill=metro_teal, minimum size=5pt, inner sep=0pt] (b) at (B) [label=above:$B$]{};
	\node[circle, fill=metro_teal, minimum size=5pt, inner sep=0pt] (c) at (C) [label=above:$C$]{};
	\node[circle, fill=metro_teal, minimum size=5pt, inner sep=0pt] (d) at (D) [label=left:$D$]{};
	\node[circle, fill=metro_teal, minimum size=5pt, inner sep=0pt] (e) at (E) [label=below:$E$]{};
	\node[circle, fill=metro_teal, minimum size=5pt, inner sep=0pt] (f) at (F) [label=below:$F$]{};

	\node[circle, fill=metro_teal, minimum size=5pt, inner sep=0pt] (g) at (G) [label=below:$G$]{};
	\node[circle, fill=metro_teal, minimum size=5pt, inner sep=0pt] (h) at (H) [label=below:$H$]{};
	\node[circle, fill=metro_teal, minimum size=5pt, inner sep=0pt] (i) at (I) [label=above:$I$]{};

	\coordinate (A) at (9,0);
	\coordinate (B) at (8.5,0.86602540378);
	\coordinate (C) at (7.5,0.86602540378);
	\coordinate (D) at (7,0);
	\coordinate (E) at (7.5,-0.86602540378);
	\coordinate (F) at (8.5,-0.86602540378);

	\coordinate (G) at (10.5, 0);
	\coordinate (H) at (11.5, 0);
	\coordinate (I) at (11.5, 1);
	
	\draw[draw=edge_gray] (A) -- (B) -- (C)  -- (D) -- (E) -- (F) -- cycle;
	
	\draw[draw=edge_gray] (A) -- (C) -- (E) -- cycle;
	\draw[draw=edge_gray] (B) -- (D) -- (F) -- cycle;
	\draw[draw=edge_gray] (A) -- (D);
	\draw[draw=edge_gray] (B) -- (E);
	\draw[draw=edge_gray] (F) -- (C);

	\draw[draw=edge_gray] (G) -- (H);
	\draw[draw=edge_gray] (H) -- (I);

	\node[circle, fill=metro_teal, minimum size=5pt, inner sep=0pt] (a) at (A) [label=right:$A$]{};
	\node[circle, fill=metro_teal, minimum size=5pt, inner sep=0pt] (b) at (B) [label=above:$B$]{};
	\node[circle, fill=metro_teal, minimum size=5pt, inner sep=0pt] (c) at (C) [label=above:$C$]{};
	\node[circle, fill=metro_teal, minimum size=5pt, inner sep=0pt] (d) at (D) [label=left:$D$]{};
	\node[circle, fill=metro_teal, minimum size=5pt, inner sep=0pt] (e) at (E) [label=below:$E$]{};
	\node[circle, fill=metro_teal, minimum size=5pt, inner sep=0pt] (f) at (F) [label=below:$F$]{};

	\node[circle, fill=metro_teal, minimum size=5pt, inner sep=0pt] (g) at (G) [label=below:$G$]{};
	\node[circle, fill=metro_teal, minimum size=5pt, inner sep=0pt] (h) at (H) [label=below:$H$]{};
	\node[circle, fill=metro_teal, minimum size=5pt, inner sep=0pt] (i) at (I) [label=above:$I$]{};

	\pgfsetlinewidth{.3ex} 

	\draw [->,>=latex,draw=edge_gray] (4.664,0) -- (5.664,0);

\end{tikzpicture}